\documentclass[3p,times,fleqn,preprint]{elsarticle}

\usepackage{ecrc}

\usepackage{etoolbox}
\makeatletter

\def\ps@pprintTitle{%
\def\@oddfoot{\footnotesize\itshape
 \ifx\@journal\@empty Preprint submitted to Elsevier
\else\@journal\fi\hfill May 3, 2014}%
\let\@evenfoot\@oddfoot}
\makeatother
\volume{00}

\firstpage{1}

\journalname{Nonlinear Analysis Series A: Theory, Methods \& Applications}

\runauth{R. H. De Staelen and M. Slodi\v{c}ka}

\jid{jde}

\jnltitlelogo{Nonlinear Analysis Series A}

\CopyrightLine{2014}{Author's personal copy}

\usepackage{amssymb,color}
\usepackage{amsthm}
\usepackage{amsmath,mathtools}    
\usepackage[subnum]{cases}

 \usepackage{lineno}

\usepackage[figuresright]{rotating}

\definecolor{darkgreen}{rgb}{0,0.6,0}

\newcommand{\del}[1]{{\color{red}}}
\newcommand{\overslaan}[1]{}

\DeclareMathOperator{\Leb}{L}
\DeclareMathOperator{\Cont}{C}
\DeclareMathOperator{\Hi}{H}
\DeclareMathOperator{\di}{d\hspace{-1.5pt}}

\def\half{{\textstyle\frac{1}{2}}}

\newcommand {\I}{[0,T]}

\newcommand {\Iopen}{(0,T)}

\def\refe#1{(\ref{#1})}

\newcommand {\dt}{\ \di t}

\newcommand {\X}{\vector{x}}
\newcommand {\dX}{\ \di \X}
\newcommand {\var}{(\X,t)}

\newcommand {\abs}[1]{\left|#1\right|}
\newcommand {\vnorma}[1]{\left\|#1\right\|}

\newcommand {\D} {\displaystyle}

\newcommand {\weak } {\rightharpoonup }

\newcommand {\normal}{\vector{\nu}}

\newcommand {\domain}{\Omega}

\newcommand {\conTk}[1]{\Cont^{#1}(\I)}

\newcommand {\hk}[1]{\Hi^{#1}(\domain )}

\newcommand {\lp}[1]{\Leb^{#1} (\domain )}

\newcommand {\lpG}[1]{\Leb^{#1} (\Gamma)}

\newcommand {\ckIX}[2]{\Cont^{#1}\left(\I,#2\right)}
\newcommand {\ckIlp}[2]{\Cont^{#1}\left(\I,\lp{#2}\right)}
\newcommand {\ckIlpG}[2]{\Cont^{#1}\left(\I,\lpG{#2}\right)}

\newcommand {\scal}[2]{\left(#1,#2\right)}

\def \vector#1{\mathbf{#1}}

\def\qqed{{\ifhmode\unskip\nobreak\fi%
         \ifvmode\vskip-\the\normalbaselineskip\leavevmode\unskip\nobreak\fi%
         \leaders\hbox to 1em{}\hfill%
         \ifmmode\clubsuit\else$\mathsurround0pt\square$\fi%
         }}

\def\mt{t\kern-0.035cm\char39\kern-0.03cm}
\def\ml{l\kern-0.035cm\char39\kern-0.03cm}
\def\md{d\kern-0.035cm\char39\kern-0.03cm}
\def\mL{L\kern-0.035cm\char39\kern-0.03cm}

\newcommand {\veps} {{\varepsilon}}

\newcommand {\NN } {{\mathbb N}}
\newcommand {\RR } {{\mathbb R}}

\newcommand {\drm}{{\rm d}}

\newtheorem{stel}{Theorem}

\newtheorem{lem}[subsection]{Proposition}

\begin{document}

\begin{frontmatter}



\dochead{PREPRINT}

\title{Reconstruction of a convolution kernel in a semilinear parabolic problem based on a global measurement}

\author{R. H. De Staelen}
\author{M. Slodi\v{c}ka}
\address{Department of Mathematical Analysis, research group of Numerical Analysis and Mathematical Modeling (NaM$^\text{\it2}$\hspace{-0.5mm}),\\ Ghent University, Galglaan 2 - S22, Gent 9000, Belgium}
\ead{rob.destaelen@ugent.be}
\ead[url]{http://cage.ugent.be/~rds}
\ead{marian.slodicka@ugent.be}
\ead[url]{http://cage.ugent.be/~ms}


\begin{abstract}
A semilinear parabolic problem of second order with an unknown time-convolution kernel is considered. The missing  kernel is recovered from an additional  integral measurement. The existence, uniqueness and  regularity  of a weak solution is addressed. 
We design a numerical algorithm based on Rothe's method, derive  a priori estimates and prove convergence of iterates towards the exact solution.
\end{abstract}


\end{frontmatter}


\section{Introduction}
\label{sec:introduction}

The general nature of an inverse problem (IP) is to deduce a cause from an effect.  
IPs typically lead to mathematical models that are ill-posed in the sense of Hadamard -- see \cite{hadamard.53}.
   Moreover, ill-posed problems frequently turn out to be numerically unstable (sensitive to small errors in the known data), in that small changes in the known data may lead to arbitrarily large changes in the response. 
   Many IPs do not have a solution in the strict classical sense, or if there is a solution, it might not be unique or might not depend continuously on the data. To obtain global in time existence and
uniqueness of a solution is in general a very  difficult part of the problem. The second important component of the task is to describe a constructive way how to find the solution. 
The usual algorithms start with parametrization of the problem and they make use of  continuous dependence of a parametrized solution on the parameter. An error/cost functional is constructed and minimized  in suitable function spaces linked to the setting under consideration. 
The bottleneck of this approach is convexity of the functional, caused by ill-posedness of the IP. In most cases the missing convexity is remediated by an appropriate regularization cf. e.g. \cite{engl,rieder,isakov}. The Tikhonov-regularization is based on adding a suitable term to 
the functional in order to guarantee its convexity, ensuring the existence of a unique solution to the minimization problem. 
This later problem can be solved numerically by adequate approximation techniques, such as the steepest descend, Ritz or Newton or Levenberg-Marquardt method, see e.g.  \cite{miller, pedegral}.

In this paper, we are interested in determining of the unknown couple $\langle u,K\rangle$ obeying the following semilinear parabolic problem
\begin{equation}
	\label{problem}
	\left\{\begin{array}{lcl}
	\partial_t u\var -\Delta u\var +K(t)h\var +(K\ast u(\X))(t)=f(\X,t,u\var ,\nabla u\var ),\quad \text{in } \domain\times I,
	\\[.2cm]
	-\nabla u\var\cdot \normal=g\var,  \quad \text{on } \Gamma\times I,\\[.2cm]
	u(\vector x,0)=u_0(\vector x),\quad \text{in } \domain,
	\end{array}\right.
\end{equation}
where $\domain$ is a Lipschitz domain (cf. \cite{Kufner}) in $\RR^N$, $N\ge 1$, with $\partial\domain=\Gamma$ and $I=\I$, $T>0$ in the time frame. By $K\ast u$ we denote the usual convolution in time, namely 
$\D (K\ast u(\X))(t)=\int_0^t K(t-s)u(\X,s)\di s$ . The missing  time-convolution kernel $K=K(t)$ will be recovered from the following integral-type measurement
\begin{equation}
	\label{problem.K}
	\int_\domain u(\vector x,t)\dX   =m(t),\qquad t\in\I.
\end{equation}
The integral  type over-determination in IPs combined with evolutionary PDEs  has been studied in several papers, e.g. \cite{prilepko.orlovsky.vasin,lesnic_2011,slod13.6} and the references therein.

Such type of integro-differential problems arise for example elastoplasticity (cf. \cite{renardy_hrusa_nohel_1987}) or in the theory of reactive contaminant transport. In \cite{delleur} one considers the following differential equation 
\[\partial_t C+\nabla\cdot(\vector{V}C)-\Delta C=\frac{-\rho_b}{n}\partial_t S\]
for the aqueous concentration $C$ and sorbed concentration per unit mass of solid $S$
with mass transformation rate in first order kinetics form of
\[\partial_t S=K_r(K_d C-S)\]
with desorption rate $K_r$ and equilibrium distribution coefficient $K_d$. This is indeed a problem of type \refe{problem} for $u=C$ with 
$K(t)=-\frac{\rho_b}{n}K_r^2K_de^{-K_rt}$, $h(t)=-\frac{S_0}{K_r K_d}$ and $f(x,\vector r)=\frac{-\rho_b}{n}K_rK_d x -\vector{V}\cdot\vector r$.

Identification of missing memory kernels in evolutionary PDEs is relatively new in IPs. We are aware of only a few papers dealing with this topics, namely \cite{colombo_guidetti_lorenzi_2003,colombo_guidetti_vespri_2006,colombo_guidetti_2007,guidetti_2007,colombo_guidetti_2011}. 
In \cite{colombo_guidetti_vespri_2006} a global in time existence and uniqueness result for an inverse problem arising in the theory of heat conduction for materials with memory has been studied. The reference \cite{colombo_guidetti_2011} derives some local and global in time existence results for 
the recovery of memory kernels. There is no description of constructive algorithms how to find a solution. 

The main goal of this paper is to design a productive numerical scheme describing the way of retrieving the couple $\langle u,K\rangle$. This is achieved not by minimization of a cost functional (which is typical for IPs) but on the time discretization based on Rothe's method \cite{rektorys,kacur}.
First, we start with derivation of a suitable variational formulation. Section \ref{sec:stability} is devoted to the study of regularity of a weak solution, and the uniqueness is addressed in Theorem \ref{thm:uniqueness}. Section \ref{sec:discretization} deals with a time discretization, 
where (based on backward Euler scheme) the continuous problem is approximated by a sequence of steady state settings at each point of a time partitioning. Stability analysis of approximates is performed in appropriate function spaces 
and  convergence (based on compactness argument) is established in Theorem \ref{thm:existence}.

\paragraph{Notations}
Denote by $\scal{\cdot}{\cdot}$ the standard inner product of $\lp{2}$ and $\vnorma{\cdot}$ its induced norm. When working at the boundary $\Gamma$ we use a similar notation, namely 
$\scal{\cdot}{\cdot}_\Gamma$,  $\lpG{2}$ and $\vnorma{\cdot}_\Gamma$. By $\ckIX{}{X}$ we denote the set of abstract functions $w:\I\to X$ equipped with the usual norm $\max_{t\in\I}\vnorma{\cdot}_X$ and 
$\Leb^p\left(\Iopen,X\right)$ is furnished with the norm $\D\left(\int_0^T\vnorma{\cdot}_X^p\dt\right)^{\frac 1p}$ with $p>1$, cf. \cite{gajewski}.
The symbol $X^*$ stands for the dual space to $X$.

We take a test function $\phi\in\hk{1}$, and derive from \refe{problem} after integration over $\domain$ that
\begin{equation}
	\label{Pvar}
	\scal{\partial_t u}{\phi}-\scal{\Delta u}{\phi}+K\scal{h}{\phi}+\scal{K\ast u}{\phi}=\scal{f(u,\nabla u)}{\phi}.
\end{equation}
Make use of Green's first identity to obtain
\begin{equation}
	\tag{P}\label{P}
	\scal{\partial_t u}{\phi}+\scal{\nabla u}{\nabla\phi}+\scal{g}{\phi}_{\Gamma}+K\scal{h}{\phi}+\scal{K\ast u}{\phi}=\scal{f(u,\nabla u)}{\phi},
\end{equation}
If we set $\phi=1$ in \refe{P}   we obtain together with the measurement $(u(t),1)=m(t)$ that
\[m'+(g,1)_{\Gamma}+K(h,1)+K\ast m=(f(u,\nabla u),1).\tag{MP}\label{MP}\]
The relations \refe{P} and \refe{MP} represent the variational formulation of \refe{problem} and \refe{problem.K}.

Finally, as is usual in papers of this sort, $C , \veps$ and $C_\veps$ will denote generic positive constants depending only on a priori
known quantities, where $\veps$ is small and $C_\veps=C\left(\veps^{-1}\right)$ is large.

\section{Stability analysis of a solution, uniqueness}
\label{sec:stability}

First, we start with a study of natural regularity of a solution $\langle u,K\rangle$. This helps us to choose  appropriate function spaces for the variational framework. Uniqueness of a solution is addressed at the end of this section.
\begin{lem}
	\label{K_bound} 
	Let $f$ be bounded, i.e. $|f|\le C$. Moreover assume that $u_0\in\lp{2}$, $g\in\ckIlpG{}{2}$, $h\in\ckIlp{}{2}$, $ \min_{t\in\I} | (h(t),1)| \ge\omega>0$ and $m\in\conTk{1}$. If $\langle u,K\rangle$ is a solution of \refe{problem} and \refe{problem.K}, then $K$ is bounded on $\I$, i.e.
	\[\max_{t\in\I}|K(t)|\le C.\]
\end{lem}
\begin{proof}
Take any $t\in\I$. 
From \refe{MP} it follows that  
\[\abs{K(t)(h(t),1)}\le  \abs{(f(u(t),\nabla u(t)),1)}+\abs{(K\ast m)(t)}+\abs{m'(t)}+\abs{(g(t),1)_{\Gamma}}.\]
Involving the assumptions we see that 
\[\omega \abs{K(t)} \le \abs{(h(t),1)}\abs{K(t)}\le  C+\abs{(K\ast m)(t)}\le C+C\int_0^t\abs{K(s)}\di s.\]
We conclude the proof by Gr\"onwall's argument, cf. \cite{bainov}.
\end{proof}

\begin{lem}
	\label{u_bound} 
	Let the conditions of Proposition~\ref{K_bound} be satisfied.  If $\langle u,K\rangle$ is a  solution of \refe{problem} and \refe{problem.K}, then there exists $C>0$   such that
	\begin{itemize}
		\item [(i)] $\D\max_{t\in\I}\vnorma{u(t)}^2+\int_0^T\vnorma{\nabla u(\xi)}^2\drm \xi\le  C$
		\item[(ii)] $\D\int_0^T\vnorma{\partial_t u}^2_{\left(\hk{1}\right)^*}\le C$.
	\end{itemize}
\end{lem}
\begin{proof}
$(i)$
If we set $\phi=u$ in \refe{P} and integrate in time over $(0,t)$ we obtain 
\begin{equation}
	\int_0^t(\partial_t u,u)\drm \xi+\int_0^t(\nabla u,\nabla u)\drm \xi+\int_0^t(g,u)_{\Gamma}\drm \xi+\int_0^tK(h,u)\drm \xi+\int_0^t(K\ast u,u)\drm \xi=\int_0^t(f(u,\nabla u),u)\drm \xi.
	\label{Pu}
\end{equation}
The first two terms can be rewritten as 
\[\int_0^t(\partial_t u,u)\drm \xi=\frac{1}{2}\vnorma{u(t)}^2-\frac{1}{2}\vnorma{u_0}^2,\quad \int_0^t(\nabla u,\nabla u)\drm \xi=\int_0^t\vnorma{\nabla u(\xi)}^2\drm \xi.\]
For the third one we get
\[
	\abs{\int_0^t(g,u)_{\Gamma}\drm \xi}
	\le  \int_0^t\vnorma{g}_{\Gamma}\vnorma{u}_{\Gamma}\drm \xi
	\le  C\int_0^t\vnorma{g}_{\Gamma}\vnorma{u}_{\hk{1}}\drm \xi
	\le  C_\veps\int_0^t\vnorma{g}^2_{\Gamma}+\veps\int_0^t\vnorma{u}^2_{\hk{1}}\drm \xi
\]
by Cauchy's inequality, the trace theorem and Young's inequality. The fourth term is easily bounded by
\[\abs{\int_0^tK(h,u)\drm \xi}\le  \int_0^t\abs{K}\vnorma{h}\vnorma{u}\drm \xi\le  C \int_0^t\vnorma{h}^2\drm \xi+ C \int_0^t\vnorma{u}^2\drm \xi,\]
as $K$ is bounded, see Proposition~\ref{K_bound}. It holds 
\begin{equation}
	\label{eq:convolution}
	\vnorma{(K\ast u)(t)}^2=\int_\domain \left(\int_{0}^{t}K(t-s)u(\X ,s)\di s \right)^2\dX \le  \int_\domain \int_{0}^{t}K^2(t-s)\int_{0}^{t}u^2(\X,s)\di s\dX \le C \int_0^t \vnorma{u(s)}^2\di s.
\end{equation}
 The last term in the left-hand side of \refe{Pu} is
\[
	\abs{\int_0^t(K\ast u,u)\drm \xi}\le  \int_0^t\vnorma{K\ast u}\vnorma{u}\drm \xi
	\le  \frac{1}{2}\int_0^t\vnorma{K\ast u}^2\drm \xi+\frac{1}{2}\int_0^t\vnorma{u}^2\drm \xi \le   C \int_0^t\vnorma{u}^2\drm \xi.
\]
 The right-hand side of \refe{Pu} can be estimated as follows 
\[
	\abs{\int_0^t(f(u,\nabla u),u)\drm \xi}\le  \int_0^t\vnorma{f(u,\nabla u)}\vnorma{u}\drm \xi
	\le  \frac{1}{2}\int_0^t\vnorma{f(u,\nabla u)}^2\drm\xi+\frac{1}{2}\int_0^t\vnorma{u}^2\drm\xi
	\le  C+\frac{1}{2}\int_0^t\vnorma{u}^2\drm\xi,
\]
as $f$ is bounded. 

Putting all things together, fixing a sufficiently small $\veps>0$ and taking into account  $\vnorma{u}^2_{\hk{1}}=\vnorma{u}^2+\vnorma{\nabla u}^2$ we obtain
\[\vnorma{u(t)}^2+\int_0^t\vnorma{\nabla u(\xi)}^2\drm \xi\le  C+ C \int_0^t\vnorma{u}^2\drm \xi,\]
which is valid for any $t\in\I$. An application of  Gr\"onwall's lemma concludes the proof.

$(ii)$
Starting from \refe{P} and using the Cauchy inequality, Lemma \ref{K_bound}, \refe{eq:convolution}, trace theorem   and  Lemma \ref{u_bound}$(i)$ we successively deduce that
\[
	\begin{array}{rlll}\D
		\abs{\scal{\partial_t u}{\phi}} &\D=  \abs{\scal{f(u,\nabla u)}{\phi} -\scal{\nabla u}{\nabla\phi}-\scal{g}{\phi}_{\Gamma}-K\scal{h}{\phi}+\scal{K\ast u}{\phi}}
		\\&\D\le C\left( \vnorma{\phi} +\vnorma{\nabla u}\vnorma{\nabla\phi} +\vnorma{\phi}_\Gamma + \sqrt{\int_0^t\vnorma{u}^2}\ \vnorma{\phi}\right)
		\\&\D\le C\left( \vnorma{\nabla u}\vnorma{\nabla\phi}+\vnorma{\phi}_{\hk{1}}\right).
	\end{array}
\]
Thus $ \scal{\partial_t u}{\phi}$ can be seen as a linear functional on $\hk{1}$ and we may write
\[
	\vnorma{\partial_t u}_{\left(\hk{1}\right)^*}=\sup_{\vnorma{\phi}_{\hk{1}}\le 1} \abs{\scal{\partial_t u}{\phi}}\le C \left( 1+ \vnorma{\nabla u}\right),
\]
which implies that 
\[
	\int_0^T\vnorma{\partial_t u}^2_{\left(\hk{1}\right)^*}\le C + C \int_0^T \vnorma{\nabla u}^2\drm \xi\le C.
\]
\end{proof}

\begin{lem}
	\label{gradu_bound} 
	Let the conditions of Proposition~\ref{K_bound} be satisfied and moreover $g\in\ckIlpG{1}{2}$  and $u_0\in\hk{1}$. 
	If $\langle u,K\rangle$ is a  solution of \refe{problem} and \refe{problem.K}, then  there exists $C>0$ such that
	\[\max_{t\in\I}\vnorma{\nabla u(t)}^2+\int_0^T\vnorma{\partial_t u(\xi)}^2\drm \xi\le  C.\]
\end{lem}
\begin{proof}
If we set $\phi=\partial_t u$ in \refe{P} and integrate in time we obtain 
\begin{equation}
	\int_0^t(\partial_t u,\partial_t u)\drm \xi+\int_0^t(\nabla u,\nabla \partial_t u)\drm \xi+\int_0^t(g,\partial_t u)_{\Gamma}\drm \xi+\int_0^tK(h,\partial_t u)\drm \xi+\int_0^t(K\ast u,\partial_t u)\drm \xi=\int_0^t(f(u,\nabla u),\partial_t u)\drm \xi.
	\label{Put}
\end{equation}
The first two terms can be rewritten as 
\[\int_0^t(\partial_t u,\partial_t u)\drm \xi=\int_0^t\vnorma{\partial_tu(\xi)}^2\drm\xi,\quad \int_0^t(\nabla u,\nabla \partial_t u)\drm \xi=\frac{1}{2}\vnorma{\nabla u(t)}^2 -\frac{1}{2}\vnorma{\nabla u_0}^2 .\]
For the third one we first integrate by parts,
\[\int_0^t(g,\partial_t u)_{\Gamma}\drm \xi=(g(t),u(t))_{\Gamma}-(g(0),u_0)_{\Gamma}-\int_0^t(\partial_t g, u)_{\Gamma}\drm \xi\]
and get
\[\abs{\int_0^t(g,\partial_t u)_{\Gamma}\drm \xi}\le  \vnorma{g(t)}_\Gamma \vnorma{u(t)}_\Gamma +\vnorma{g(0)}_\Gamma \vnorma{u_0}_\Gamma +\int_0^t \vnorma{\partial_t g}_\Gamma \vnorma{u}_\Gamma \drm \xi\]
\[\le  C_{\veps}+\veps\vnorma{u}^2_{\hk{1}}+C\int_0^t \vnorma{u}^2_{\hk{1}}\drm \xi\]
by Cauchy's inequality, the trace theorem and Young's inequality. 
The fourth term is easily bounded by
\[\abs{\int_0^tK(h,\partial_tu)\drm \xi}\le  \int_0^t\abs{K}\vnorma{h}\vnorma{\partial_tu}\drm \xi\le  C_{\veps} \int_0^t\vnorma{h}^2\drm \xi+ \veps \int_0^t\vnorma{\partial_tu}^2\drm \xi,\]
as $K$ is bounded, see Proposition~\ref{K_bound}.
The last term in the left-hand side of \refe{Put} can be estimated using \refe{eq:convolution} and  Proposition~\ref{u_bound} as follows 
\[\abs{\int_0^t(K\ast u,\partial_tu)\drm \xi}\le  \int_0^t\vnorma{K\ast u}\vnorma{\partial_tu}\drm \xi
\le  C_{\veps} \int_0^t\vnorma{K\ast u}^2\drm \xi+\veps\int_0^t\vnorma{\partial_tu}^2\drm \xi \le C_\veps +  \veps\int_0^t\vnorma{\partial_tu}^2\drm \xi. \] 
For the right-hand side of \refe{Put} we deduce that 
\[\abs{\int_0^t(f(u,\nabla u),\partial_tu)\drm \xi}\le  \int_0^t\vnorma{f(u,\nabla u)}\vnorma{\partial_tu}\drm \xi
\le  C_{\veps} +\veps\int_0^t\vnorma{\partial_tu}^2\drm\xi,  \]
as $f$ is bounded. 

Putting  things together we arrive at 
 \[\left(\half-\veps\right)\vnorma{\nabla u(t)}^2+(1-\veps)\int_0^t\vnorma{\partial_tu(\xi)}^2\drm \xi \le  C_{\veps}+ C\int_0^t \vnorma{\nabla u}^2\drm \xi,\]
which is valid for any $t\in\I$. Fixing a suitable $\veps>0$  we conclude the proof by Gr\"onwall's lemma.
\end{proof}

\begin{lem}
	\label{lapu_bound} 
	Let the conditions of Proposition~\ref{K_bound} be satisfied. Moreover assume that $g\in\ckIlpG{1}{2}$, $h\in\ckIX{}{\hk{1}}$,   $f$ is Lipschitz continuous in all variables, and $u_0\in\hk{2}$. If $\langle u,K\rangle$ is a solution of \refe{problem} and \refe{problem.K}, then   there exists $C>0$ such that
	\begin{itemize}
	\item[(i)]
	$\D \max_{t\in\I}\vnorma{\Delta u(t)}^2+\int_0^T\vnorma{\nabla \partial_t u}^2\drm \xi\le  C$
	\item[(ii)] $\D\max_{t\in\I}\vnorma{\partial_t u(t)} \le  C.$
	\end{itemize}
\end{lem}
\begin{proof}
$(i)$ 
If we set $\phi=-\Delta \partial_t u$ in \refe{Pvar} and integrate in time we obtain 
\begin{equation}
	-\int_0^t(\partial_t u,\Delta\partial_t u)\drm \xi+\int_0^t(\Delta u, \Delta \partial_t u)\drm \xi-\int_0^tK(h,\Delta\partial_t u)\drm \xi-\int_0^t(K\ast u,\Delta\partial_t u)\drm \xi=-\int_0^t(f(u,\nabla u),\Delta\partial_t u)\drm \xi.
	\label{Plaput}
\end{equation}
The first two terms can be rewritten as 
\[
	-\int_0^t(\partial_t u,\Delta\partial_t u)\drm \xi=\int_0^t\vnorma{\nabla\partial_tu}^2\drm\xi+\int_0^t (\partial_tu,\partial_t g)_\Gamma \drm\xi,
	\qquad \int_0^t(\Delta u,\Delta \partial_t u)\drm \xi=\frac{1}{2}\vnorma{\Delta u(t)}^2 -\frac{1}{2}\vnorma{\Delta u_0}^2 .
\]
Making use of the Cauchy, Young inequalities, the trace theorem and Proposition \ref{gradu_bound} we deduce that 
\[
	\begin{array}{rlll}\D
		\abs{\int_0^t (\partial_tu,\partial_t g)_\Gamma \drm\xi}&\D \le  \int_0^t \vnorma{\partial_tu}_\Gamma \vnorma{\partial_t g }_\Gamma  \drm\xi
		\\&\D\le \veps \int_0^t \vnorma{\partial_tu}_\Gamma^2 \drm\xi + C_\veps \int_0^t  \vnorma{\partial_t g }_\Gamma^2 \drm\xi
		\\&\D\le  \veps \int_0^t \vnorma{\partial_tu}_{\hk{1}}^2 \drm\xi + C_\veps 
		\\&\D \le  \veps \int_0^t \vnorma{\nabla \partial_t u}^2 \drm\xi + C_\veps .
	\end{array}
\]

For the third term in \refe{Plaput}  we first use the Green formula 
\[-\int_0^tK(h,\Delta\partial_t u)\drm \xi=\int_0^tK\left[(\nabla h,\nabla\partial_t u)-(h,\partial_tg)_\Gamma \right]\drm \xi\]
and get by Cauchy's and Young's inequality
\[
	\abs{\int_0^tK(h,\Delta\partial_t u)\drm \xi}
	\le  C_{\veps}\int_0^t\vnorma{\nabla h}^2\drm\xi+{\veps}\int_0^t\vnorma{\nabla\partial_t u}^2\drm\xi+C\int_0^t\left(\vnorma{h}^2_\Gamma + \vnorma{\partial_tg}^2_\Gamma \right)\drm \xi 
	\le  C_{\veps} +{\veps}\int_0^t\vnorma{\nabla\partial_t u}^2\drm\xi
\]
as $K$ (see Proposition~\ref{K_bound}) is bounded and $\vnorma{h}^2_\Gamma $ is finite by the trace theorem. The last term in the left-hand side of \refe{Plaput} is rewritten as
\[-\int_0^t(K\ast u,\Delta\partial_tu)\drm \xi= \int_0^t\left[(K\ast\nabla u,\nabla\partial_tu)-(K\ast u,\partial_tg)_\Gamma \right]\drm \xi,\]
which gives
\[\abs{\int_0^t(K\ast u,\Delta\partial_tu)\drm \xi}\le  \int_0^t\vnorma{K\ast \nabla u}\vnorma{\nabla\partial_tu}\drm \xi +  \int_0^t\vnorma{K\ast u}_\Gamma \vnorma{\partial_tg}_\Gamma \drm \xi\]
\[
	\le  C_{\veps}\int_0^t\vnorma{K\ast \nabla u}^2\drm \xi +{\veps}\int_0^t\vnorma{\nabla\partial_tu}^2\drm \xi
	+  C\int_0^t\vnorma{K\ast u}_\Gamma ^2\drm\xi+C\int_0^t\vnorma{\partial_tg}^2_\Gamma \drm \xi
 \]
\[\le  C_{\veps} +{\veps}\int_0^t\vnorma{\nabla\partial_tu}\drm \xi
\]
as $\D\vnorma{(K\ast \nabla u)(t)}^2\le  C\int_0^t\vnorma{\nabla u}^2\di s$ and $\D\vnorma{(K\ast u)(t)}_\Gamma ^2\le  C\int_0^t \vnorma{u}_\Gamma ^2\di s$, like in \refe{eq:convolution}.  The right-hand side of \refe{Plaput} is rewritten by integrating by parts as
\[-\int_0^t(f(u,\nabla u),\Delta\partial_t u)\drm \xi=\int_0^t(\partial_tf(u,\nabla u),\Delta u)\drm \xi+(f(u(0),\nabla u(0)),\Delta u(0))-(f(u(t),\nabla u(t)),\Delta u(t) )\]
so
\[
	\abs{\int_0^t(f(u,\nabla u),\Delta\partial_tu)\drm \xi}
	\le  \veps\int_0^t\vnorma{\partial_tf(u,\nabla u)}^2\drm \xi+C_{\veps}\int_0^t\vnorma{\Delta u}^2\drm \xi+C_{\veps}+\veps\vnorma{\Delta u(t)}^2
\]
\[\le \veps\int_0^t\vnorma{\partial_t\nabla u}^2\drm \xi+C_{\veps}\int_0^t\vnorma{\Delta u}^2\drm \xi+C_{\veps}+\veps\vnorma{\Delta u(t)}^2,\]
as $u_0\in\hk{2}$, $\partial_tf(u,\nabla u)=\nabla f(u,\nabla u)\cdot\langle\partial_t u,\partial_t\nabla u\rangle$, $f$ is Lipschitz in all variables and $\D\int_0^T\vnorma{\partial_t u}^2\di s$ is  bounded by Proposition~\ref{gradu_bound}.

Putting all things together we obtain
\[(1-\veps)\int_0^t\vnorma{\nabla\partial_tu(\xi)}^2\drm\xi+\left(\frac{1}{2}-\veps\right)\vnorma{\Delta u(t)}^2 \le  C_{\veps} +C_{\veps}\int_0^t\vnorma{\Delta u}^2\drm \xi,\]
which is valid for any $t\in\I$. Fixing a sufficiently small $\veps>0$ and involving Gr\"onwall's argument, we obtain the desired result. 

$(ii)$
The assertion follows readily from \refe{problem} and the already obtained stability results, i.e.
\[\vnorma{\partial_t u} = \vnorma{ \Delta u -Kh -K\ast u + f(u ,\nabla u )}\le C\left(\max_{t\in \I} |K(t)|\right) \left( 1 + \vnorma{\Delta u} + \vnorma{u}\right)\le C.\]
\end{proof}

\begin{lem}
	\label{Kt_bound} 
	Let the conditions of Proposition~\ref{K_bound} be satisfied. Moreover assume that $g\in\ckIlpG{1}{2}$, $h\in\ckIX{1}{\lp{2}}\cap \ckIX{}{\hk{1}}$,   $m\in C^2(\I)$, $f$ is Lipschitz continuous in all variables, and $u_0\in\hk{2}$. 
	If $\langle u,K\rangle$ is a solution of \refe{problem} and \refe{problem.K}, then   there exists $C>0$ such that
	\[\int_0^T \abs{K'(s)}^2 \di s\le C.\]
\end{lem}
\begin{proof}
We take the time derivative of \refe{MP} and it follows that for any time $t\in\I$ it holds
\[m''+(\partial_tg,1)_{\Gamma}+K'(h,1)+K(\partial_th,1)+Km(0)+K\ast m'=(\nabla f(u,\nabla u)\cdot\langle\partial_t u,\partial_t\nabla u\rangle,1).\tag{MP'}\label{MP'}\]
From this we infer
\[\abs{(h,1)}\abs{K'(t)}\le  \abs{(\nabla f(u,\nabla u)\cdot\langle\partial_t u,\partial_t\nabla u\rangle,1)}+\abs{K\ast m'}+C\]
as $K$ is bounded, $\partial_t h\in\ckIX{}{\lp{2}}$, $\partial_t g\in\ckIX{}{\lpG{2}}$ and $m\in \Cont^2(\I)$.  Since $f$ is Lipschitz continuous in all variables and $\partial_t u$ is $\lp{2}$-bounded we obtain
 \[ \omega\abs{K'(t)}\le \abs{(h,1)}\abs{K'}\le    C+C\vnorma{\partial_t\nabla u}.\]  
Taking square and integrating in time we arrive at 
\[ \int_0^T\abs{K'(\xi)}^2\drm\xi\le  C+C\int_0^T\vnorma{\partial_t\nabla u(\xi)}^2\drm\xi\le  C.\]
\end{proof}


\paragraph {Uniqueness} Now, we are in a position to state unicity of solution.  Suppose $\langle u_1,K_1\rangle$ and $\langle u_2,K_2\rangle$ solve \refe{P}-\refe{MP}, then
by subtracting the corresponding variational formulations from each other we obtain
\[(\partial_t (u_1-u_2),\phi)+(\nabla (u_1-u_2),\nabla\phi)+(K_1(t)-K_2(t))(h,\phi)+(K_1\ast u_1-K_2\ast u_2,\phi)=(f(u_1,\nabla u_1)-f(u_2,\nabla u_2),\phi),\]
\[(K_1(t)-K_2(t))(h,1)+(K_1-K_2)\ast m=(f(u_1,\nabla u_1)-f(u_2,\nabla u_2),1).\]
This we rewrite using $e_K(t)=K_1(t)-K_2(t)$ and $e_u(\vector x,t)=u_1(\vector x,t)-u_2(\vector x,t)$
\begin{numcases}{}
	(\partial_t e_u,\phi)+(\nabla e_u,\nabla\phi)+e_K(h,\phi)+(K_1\ast e_u,\phi)+(e_K\ast u_2,\phi)=(f(u_1,\nabla u_1)-f(u_2,\nabla u_2),\phi) \label{u1u2}\\
	e_K(h,1)+e_K\ast m=(f(u_1,\nabla u_1)-f(u_2,\nabla u_2),1). 
	\label{Ku1u2}
\end{numcases}

\begin{stel}
	\label{thm:uniqueness}
	Assume that   $h\in\ckIX{}{\lp{2}}$, $ \min_{t\in\I} | (h(t),1)| \ge\omega>0$ and $m\in\conTk{}$. 
	The function  $f$ is supposed to be  Lipschitz continuous in all variables. Then the problem \refe{P}-\refe{MP} has at most one solution $\langle u,K\rangle\in \Leb^2\left(\Iopen,\hk{1}\right)\times \Leb^2(0,T)$  with $\partial_t u\in \Leb^2\left(\Iopen,\left(\hk{1}\right)^*\right)$. 
\end{stel}
\begin{proof}
The Lipschitz continuity of $f$, Gr\"onwall's lemma  and \refe{Ku1u2} implies 
\begin{equation}
	\label{eq:pp}
	\abs{e_K(t)}\le  C\vnorma{e_u(t)}_{\hk{1}}+ C\int_0^t \vnorma{e_u}_{\hk{1}}\drm\xi.
\end{equation}

We put $\phi=e_u$ in \refe{u1u2} and integrate in time
\[
	\begin{array}{lll}\D
		\frac{1}{2}\vnorma{e_u(t)}^2+\int_0^t\vnorma{\nabla e_u}^2\drm\xi+\int_0^te_K(h,e_u)\drm\xi+\int_0^t(K_1\ast e_u,e_u)\drm\xi+\int_0^t(e_K\ast u_2,e_u)\drm\xi
		\\\D =\int_0^t(f(u_1,\nabla u_1)-f(u_2,\nabla u_2),e_u)\drm\xi.
	\end{array}
	\]
Using Cauchy's inequality, we obtain successively the bounds
\[\int_0^t\vnorma{f(u_1,\nabla u_1)-f(u_2,\nabla u_2)}\vnorma{e_u}\drm\xi\le  {\veps}\int_0^t\vnorma{\nabla e_u}^2\drm\xi+C_{\veps}\int_0^t\vnorma{e_u}^2\drm\xi,\]
as $f$ is Lipschitz,
\[\int_0^t\vnorma{e_K\ast u_2}\vnorma{e_u}\drm\xi\le  \veps\int_0^t e_K^2\drm\xi+C_{\veps}\int_0^t\vnorma{e_u}^2\drm\xi\] 
as $u_2\in\ckIX{}{\lp{2}}$, which follows from $\partial_t u_2\in \Leb^2\left(\Iopen,\lp{2}\right)$,
\[\int_0^t\vnorma{K_1\ast e_u}\vnorma{e_u}\drm\xi\le  C\int_0^t\vnorma{e_u}^2\drm\xi,\] 
as $K_1\in \Leb^2(0,T)$, and using $h\in\ckIX{}{\lp{2}}$
\[\int_0^t\abs{e_K}\vnorma{h}\vnorma{e_u}\drm\xi\le  \veps\int_0^t\abs{e_K}^2 \drm\xi+C_{\veps}\int_0^t \vnorma{e_u}^2\drm\xi\le \veps\int_0^t\vnorma{e_u}_{\hk{1}}^2\drm\xi+C_{\veps}\int_0^t \vnorma{e_u}^2\drm\xi.\]
From these estimates we obtain
\[\vnorma{e_u(t)}^2+(1-\veps)\int_0^t\vnorma{\nabla e_u}^2\drm\xi\le  C_{\veps}\int_0^t\vnorma{e_u}^2\drm\xi,\]
and conclude that $\D\max_{t\in\I}\vnorma{e_u(t)}^2+ \int_0^T\vnorma{\nabla e_u}^2\drm\xi=0$ by Gr\"onwall's lemma when fixing a suitable $\veps>0$.  So $u$ is unique in $\ckIX{}{\lp{2}}\cap \Leb^2\left(\Iopen,\hk{1}\right)$.
The uniqueness of $K$ in $\Leb^2(0,T)$ follows from \refe{eq:pp}. 
\end{proof}


\section{Time discretization, existence of a solution}
\label{sec:discretization}
Rothe's method \cite{kacur,rektorys} represents a constructive method suitable for solving
evolution problems. Using a simple discretization in time, a
time-dependent problem is approximated by a sequence of elliptic
problems which have to be solved successively with increasing time
step. This standard technique is in our case complicated by the 
unknown convolution kernel  $K$. There exists a simple way  to overcome  this difficulty.

For ease of explanation we consider an equidistant time-partitioning of the time frame $\I$ with a step $\tau=T/n,$ for any $n\in\NN$. We use the notation $t_i=i\tau$ and for any function $z$ we write
\[z_i=z(t_i),\qquad \delta z_i = \frac {z_i-z_{i-1}}{\tau}.\]
 
  We  will consider a decoupled system with unknowns $\langle u_i,K_i\rangle$ for $i=1,\dots,n$.
At time $t_i$ we infer from \refe{Pvar} the backward Euler scheme
 \begin{equation}
	\scal{\delta u_i}{\phi}-\scal{\Delta u_i}{\phi}+K_i\scal{h_i}{\phi}+\scal{\sum_{k=1}^i K_k u_{i-k}\tau}{\phi}=\scal{f_{i-1}}{\phi}.
	\label{Pivar}
\end{equation}
where $f_i=f(u_{i},\nabla u_{i})$.
Like \refe{P} and \refe{MP} one obtains for $\phi\in\hk{1}$ that
\[
	\scal{\delta u_i}{\phi}+\scal{\nabla u_i}{\nabla\phi}+ \scal{g_i}{\phi}_\Gamma +K_i\scal{h_i}{\phi}+\scal{\sum_{k=1}^i K_k u_{i-k}\tau}{\phi}=\scal{f_{i-1}}{\phi}
	\tag{DP$i$}\label{DPi}
\]
and
\[m'_i+(g_i,1)_{\Gamma}+K_i(h_i,1)+\sum_{k=1}^i K_k m_{i-k}\tau=(f_{i-1},1).\tag{DMP$i$}\label{DMPi}\]
Note that for a given $i\in\{1,\dots,n\}$ we solve first \refe{DMPi} and then \refe{DPi}. Further we increase $i$ to $i+1$.
 
\begin{lem}
	\label{Ki_bound} 
	Let $f$ be bounded, i.e. $|f|\le C$. Moreover assume that  $g\in\ckIlpG{}{2}$, $h\in\ckIlp{}{2}$, $ \min_{t\in\I} | (h(t),1)| \ge\omega>0$, $u_0\in\hk{1}$ and $m\in\conTk{1}$. 
	Then there exist $C>0$ and $\tau_0>0$ such that for any $\tau<\tau_0$ and each $i\in\{1,\dots,n\}$ we have
	\begin{itemize}
	 \item [(i)] there exist $K_i\in\RR$ and $u_i\in\hk{1}$ obeying \refe{DMPi} and \refe{DPi}
	 \item[(ii)] $\D\max_{1\le i\le n}|K_i|\le C$.
	\end{itemize}
\end{lem}
\begin{proof}
$(i)$ Set $\D \tau_0=\min\left\{ 1, \frac{\omega}{2\abs{m_0}}\right\}$. Then for any $\tau<\tau_0$ we may write by triangle inequality that

\[ 
0< \omega - \abs{m_0}\tau_0\le \omega - \abs{m_0}\tau\le \abs{ \scal{h_i}{1}} - \abs{m_0}\tau \le \abs{\scal{h_i}{1} - m_0\tau  }
\]
We apply the following recursive deduction for $i=1,\dots,n$. 
\begin{enumerate}[Step 1:]
	\item Let $u_{i-1}\in\hk{1}$ be given. Then \refe{DMPi} implies the existence of $K_i\in\RR$ such that
	\begin{equation}\label{eq:ki}
		K_i\left[\scal{h_i}{1} - m_0\tau \right] = \scal{f_{i-1}}{1} - m'_i - \scal{g_i}{1}_\Gamma - \sum_{k=1}^{i-1} K_k m_{i-k}\tau.
	\end{equation}
	\item The existence of $u_i\in\hk{1}$ follows from \refe{DPi} by the Lax-Milgram lemma. 
\end{enumerate}

$(ii)$ The relation \refe{eq:ki} yields 
\[ 
	\abs{K_i} \le C\left( 1+ \sum_{k=1}^{i-1} \abs{K_{k}}\tau\right),
\]
which is valid for any $i=1,\dots,n$. An application of the discrete Gr\"{o}nwall lemma gives the uniform bound of $\abs{K_i} $.
\end{proof}

\begin{lem}
\label{ui_bound} 
	Let the conditions of Proposition~\ref{Ki_bound} be satisfied.  Then there exists $C>0$ such that for any $\tau<\tau_0$   
	\[\max_{1\le j\le n}\vnorma{u_j}^2+\sum_{i=1}^n\vnorma{\nabla u_i}^2 \tau +\sum_{i=1}^n\vnorma{u_i-u_{i-1}}^2\le  C.\]
\end{lem}
\begin{proof}
If we set $\phi=u_i\tau$ in \refe{DPi} and sum up for $i=1,\dots,j$ we obtain 
\begin{equation}
	\sum_{i=1}^j(\delta u_i,u_i)\tau+\sum_{i=1}^j\vnorma{\nabla u_i}^2\tau+\sum_{i=1}^j(g_i,u_i)_{\Gamma}\tau+\sum_{i=1}^jK_i(h_i,u_i)\tau+\sum_{i=1}^j\sum_{k=1}^i (K_k u_{i-k}\tau,u_i)\tau=\sum_{i=1}^j(f_{i-1},u_i)\tau.
	\label{Pui}
\end{equation}
The summation by parts formula formula says that  
\[\sum_{i=1}^j(\delta u_i,u_i)\tau=\sum_{i=1}^j(u_i-u_{i-1},u_i)=\frac{1}{2}\left(\vnorma{u_j}^2-\vnorma{u_0}^2+\sum_{i=1}^k \vnorma{u_i-u_{i-1}}^2\right).\]
For the third term of \refe{Pui} we get
\[
	\abs{\sum_{i=1}^j(g_i,u_i)_{\Gamma}\tau}\le 
	 \sum_{i=1}^j\vnorma{g_i}_\Gamma \vnorma{u_i}_\Gamma \tau
	\le  C \sum_{i=1}^j\vnorma{g_i}_\Gamma \vnorma{u_i}_{\hk{1}}\tau 
	\le  C_\veps \sum_{i=1}^j\vnorma{g_i}^2_\Gamma\tau +\veps \sum_{i=1}^j\vnorma{u_i}^2_{\hk{1}}\tau
\]
by Cauchy's inequality, the trace theorem and Young's inequality. The fourth term in \refe{Pui} is easily bounded by
\[\abs{ \sum_{i=1}^j K_i(h_i,u_i)\tau}\le   \sum_{i=1}^j\abs{K_i}\vnorma{h_i}\vnorma{u_i}\tau\le  C \sum_{i=1}^j \vnorma{h_i}^2\tau + C  \sum_{i=1}^j\vnorma{u_i}^2\tau,\]
as $K_i$ is bounded, see Proposition~\ref{Ki_bound}. The last term in the left-hand side of \refe{Pui} is
\[
	\abs{\sum_{i=1}^j\sum_{k=1}^i (K_k u_{i-k},u_i)\tau^2}
	\le  \sum_{i=1}^j\sum_{k=1}^i \abs{(K_k u_{i-k},u_i)}\tau^2 
	\le  C\sum_{i=1}^j\sum_{k=1}^i\vnorma{u_{i-k}}^2\tau^2+C\sum_{i=1}^j\sum_{k=1}^i \vnorma{u_i}^2\tau^2
	\le  C\sum_{i=0}^j\vnorma{u_i}^2\tau,\]
again as $K_i$ is bounded, see Proposition~\ref{Ki_bound}. The right-hand side of \refe{Pui} can be estimated as follows
\[\abs{\sum_{i=1}^j(f_{i-1},u_i)\tau}\le   \sum_{i=1}^j\vnorma{f_{i-1}}\vnorma{u_i}\tau
\le  C+C\sum_{i=1}^j\vnorma{u_i}^2\tau ,\]
as $f$ is bounded. 

Putting all things together we obtain
\[\vnorma{u_j}^2+\sum_{i=1}^k \vnorma{u_i-u_{i-1}}^2+(1-\veps)\sum_{i=1}^j\vnorma{\nabla u_i}^2\tau\le  C_\veps +C \sum_{i=1}^j\vnorma{u_i}^2\tau.\]
Fixing a sufficiently small $\veps>0$ and involving the discrete Gr\"{o}nwall lemma we conclude the proof. 
\end{proof}

\begin{lem}
	\label{gradui_bound} 
	Let the conditions of Proposition~\ref{Ki_bound} be satisfied. Moreover suppose that $g\in\ckIlpG{1}{2}$.
	Then there exists $C>0$  such that for any $\tau<\tau_0$  it holds 
	\[\max_{1\le j\le n}\vnorma{\nabla u_j}^2+\sum_{i=1}^n\vnorma{\delta u_i}^2\tau +\sum_{i=1}^n\vnorma{\nabla u_i-\nabla u_{i-1}}^2\le  C.\]
\end{lem}
\begin{proof}
If we set $\phi=\delta u_i\tau$ in \refe{DPi} and sum up for $i=1,\dots,j$  we obtain 
\begin{equation}
	\sum_{i=1}^j\vnorma{\delta u_i}^2\tau+\sum_{i=1}^j(\nabla u_i,\nabla\delta u_i)\tau+\sum_{i=1}^j(g_i,\delta u_i)_{\Gamma}\tau+\sum_{i=1}^jK_i(h_i,\delta u_i)\tau+\sum_{i=1}^j\sum_{k=1}^i (K_k u_{i-k},\delta u_i)\tau^2=\sum_{i=1}^j(f_{i-1},\delta u_i)\tau.
	\label{Puit}
\end{equation}
The second term can be rewritten as 
\[ 
	\sum_{i=1}^j(\nabla u_i,\nabla\delta u_i)\tau=\sum_{i=1}^j(\nabla u_i,\nabla u_i-\nabla u_{i-1})=\frac{1}{2}\left(\vnorma{\nabla u_j}^2-\vnorma{\nabla u_0}^2+\sum_{i=1}^k \vnorma{\nabla u_i-\nabla u_{i-1}}^2\right).
\]
For the third one we first use summation by parts,
\[\sum_{i=1}^j(g_i,u_i-u_{i-1})_{\Gamma}=(g_j,u_j)_{\Gamma}-(g_0,u_0)_{\Gamma}-\sum_{i=1}^j(g_i-g_{i-1},u_i)_{\Gamma}\]
and get
\[
	\begin{array}{rlll}\D
		\abs{\sum_{i=1}^j(g_i,\delta u_i)_{\Gamma}\tau}&\D \le  \vnorma{g_j}_\Gamma \vnorma{u_j}_\Gamma +\vnorma{g_0}_\Gamma \vnorma{u_0}_\Gamma +\sum_{i=1}^j\vnorma{\delta g_i}_\Gamma \vnorma{u_i}_\Gamma\tau 
		\\&\D \le  C_{\veps}+\veps\vnorma{u_j}^2_{\hk{1}}+C \sum_{i=1}^j \vnorma{u_i}^2_{\hk{1}}\tau
		\\&\D \le C_{\veps}+\veps\vnorma{\nabla u_j}^2
	\end{array}
\]
by Cauchy's inequality, the trace theorem, Young's inequality, and Proposition \ref{ui_bound}.
The fourth term in \refe{Puit} is easily bounded by
\[\abs{\sum_{i=1}^jK_i(h_i,\delta u_i)\tau}\le  \sum_{i=1}^j\abs{K_i}\vnorma{h_i}\vnorma{\delta u_i}\tau\le  C_{\veps} \sum_{i=1}^j \vnorma{h_i}^2\tau+ \veps\sum_{i=1}^j\vnorma{\delta u_i}^2\tau \le  C_{\veps}  + \veps\sum_{i=1}^j\vnorma{\delta u_i}^2\tau ,\]
as $K$ is bounded, see Proposition~\ref{Ki_bound}.
The last term in the left-hand side of \refe{Puit} can be estimated as follows
\[
	\begin{array}{rlll}\D
		\abs{\sum_{i=1}^j\sum_{k=1}^i (K_k u_{i-k},\delta u_i)\tau^2 }&\D  \le \sum_{i=1}^j\sum_{k=1}^i\abs{K_k }\vnorma{u_{i-k}}\vnorma{\delta u_i}\tau^2 
		\\&\D \le \sum_{i=1}^j\sum_{k=1}^i\left(C_{\veps}\vnorma{u_{i-k}}^2+\veps\vnorma{\delta u_i}^2\right)\tau^2 
		\\&\D \le C_{\veps}+\veps  \sum_{i=1}^j \vnorma{\delta u_i}^2\tau
	\end{array}
\]
using Propositions \ref{Ki_bound} and \ref{ui_bound}. The right-hand side of \refe{Puit} can be enlarged by 
\[
	\abs{\sum_{i=1}^j(f_{i-1},\delta u_i)\tau}\le  \sum_{i=1}^j\vnorma{f_{i-1}}\vnorma{\delta u_i} \tau \le   C_{\veps}+\veps  \sum_{i=1}^j \vnorma{\delta u_i}^2\tau 
\]
as $f$ is bounded. 

Putting all things together we obtain
\[
	(1-\veps)\sum_{i=1}^j\vnorma{\delta u_i}^2\tau +\left(\half -\veps\right) \vnorma{\nabla u_j}^2+ \half \sum_{i=1}^k \vnorma{\nabla u_i-\nabla u_{i-1}}^2\le  C_{\veps}. 
\]
  Fixing a suitable  $\veps>0$  we conclude the proof. 
\end{proof}

Inspecting the relation \refe{DPi} we may write for any $\phi\in\hk{1}$ that 
\begin{equation}
	\scal{-\Delta u_i}{\phi} = \scal{\nabla u_i}{\nabla\phi}+ \scal{g_i}{\phi}_\Gamma =\scal{f_{i-1}}{\phi} - \scal{\delta u_i}{\phi} - K_i\scal{h_i}{\phi} -\scal{\sum_{k=1}^i K_k u_{i-k}\tau}{\phi}.
	\label{eq:DPi}
\end{equation}
The term $-\Delta u_i$ has to be understood in the sense of duality, as a functional on $\hk{1}$. The right-hand side of \refe{eq:DPi} can be estimated by 
$C(1+\vnorma{\delta u_i})\vnorma{\phi}$. Thus there exists an extension of $-\Delta u_i$ to $\lp{2}$ according to Hahn-Banach theorem, cf. \cite[p. 173]{ljusternik.sobolev}. This extension will have the same norm as the functional on $\hk{1}$. 
Therefore taking into account the assumptions and the stability results from Proposition \ref{gradui_bound} we immediately obtain
\begin{equation}
	\label{eq:deltaUI}
	\sum_{i=1}^n\vnorma{\Delta u_i}^2\tau \le C + C \sum_{i=1}^n\vnorma{\delta u_i}^2\tau \le C.
\end{equation}

\begin{lem}
	\label{lapui_bound} 
	Assume that  $g\in\ckIX{1}{\lpG{2}}$, $h\in\ckIX{}{\hk{1}}$, $ \min_{t\in\I} | (h(t),1)| \ge\omega>0$, $u_0\in\hk{2}$ and $m\in\conTk{1}$. 
	The function  $f$ is supposed to be bounded, i.e. $|f|\le C$, and  Lipschitz continuous in all variables. 
	Then there exist $C>0$ such that for any $\tau<\tau_0$   we have
	\begin{itemize}
	 \item [(i)] $\D \max_{1\le j\le n}\vnorma{\Delta u_j}^2+\sum_{i=1}^n\vnorma{\nabla \delta u_i}^2\tau  +\sum_{i=1}^n\vnorma{\Delta u_i-\Delta u_{i-1}}^2\le  C$
	 \item[(ii)] $\D \max_{1\le j\le n}\vnorma{\delta u_j}\le C$.
	\end{itemize}
\end{lem}
\begin{proof}
$(i)$ 
If we set $\phi=-\Delta \delta u_i\tau$ in \refe{Pivar} and sum up for $i=1,\dots,j$  we obtain 
\begin{equation}
	-\sum_{i=1}^j(\delta u_i,\Delta \delta u_i)\tau + \sum_{i=1}^j(\Delta u_i,\Delta \delta u_i)\tau - \sum_{i=1}^jK_i(h_i,\Delta \delta u_i) \tau- \sum_{i=1}^j\sum_{k=1}^i (K_k u_{i-k},\Delta \delta u_i)\tau^2 = -\sum_{i=1}^j(f_{i-1},\Delta \delta u_i)\tau .
	\label{Plapuit}
\end{equation}
The first two terms can be rewritten as 
\[
	-\sum_{i=1}^j(\delta u_i,\Delta \delta u_i\tau)=\sum_{i=1}^j\vnorma{\nabla\delta u_i}^2\tau + \sum_{i=1}^j(\delta u_i,\delta g_i)_\Gamma\tau ,\quad \sum_{i=1}^j(\Delta u_i,\Delta \delta u_i) \tau 
	= \frac{1}{2}\left(\vnorma{\Delta u_j}^2-\vnorma{\Delta u_0}^2+\sum_{i=1}^k \vnorma{\Delta u_i-\Delta u_{i-1}}^2\right).
\]
Using the Cauchy and Young inequalities, the trace theorem and Proposition \ref{gradui_bound} we successively deduce that
\[
	\begin{array}{rlll}\D
		\abs{\sum_{i=1}^j(\delta u_i,\delta g_i)_\Gamma\tau}&\D \le \sum_{i=1}^j\vnorma{\delta u_i}_\Gamma \vnorma{\delta g_i}_\Gamma \tau
		\\&\D\le \veps \sum_{i=1}^j\vnorma{\delta u_i}_\Gamma^2\tau + C_\veps \sum_{i=1}^j\vnorma{\delta g_i}_\Gamma^2 \tau
		\\&\D\le \veps \sum_{i=1}^j\vnorma{\delta u_i}_{\hk{1}}^2\tau + C_\veps 
		\\&\D\le \veps \sum_{i=1}^j\vnorma{\delta \nabla u_i}^2\tau + C_\veps  .
	\end{array}
\]
For the third term in \refe{Plapuit} we first integrate by parts,
\[-K_i(h_i,\Delta \delta u_i)\tau = K_i\left[(\nabla h_i,\nabla\delta u_i)-(h_i,\delta g_i)_\Gamma \right] \tau\]
and get by Cauchy's and Young's inequality
\[
	\abs{\sum_{i=1}^jK_i(h_i,\Delta \delta u_i)\tau }
	\le  C_{\veps}\sum_{i=1}^j\vnorma{\nabla h_i}^2\tau + \veps\sum_{i=1}^j\vnorma{\nabla\delta u_i}^2\tau
	+ C\sum_{i=1}^j\left(\vnorma{h_i}^2_\Gamma + \vnorma{\delta g_i}^2_\Gamma \right)\tau
	\le C_{\veps} + \veps\sum_{i=1}^j\vnorma{\nabla\delta u_i}^2\tau
\]
as $K_i$ (see Proposition~\ref{Ki_bound}) is bounded, $h_i\in\hk{1}$ and $\vnorma{h_i}^2_\Gamma $ is finite by the trace theorem. The last term in the left-hand side of \refe{Plapuit} is rewritten as
\[-\sum_{i=1}^j\sum_{k=1}^i (K_k u_{i-k},\Delta \delta u_i) \tau^2 = \sum_{i=1}^j\sum_{k=1}^i\left[(K_k\nabla u_{i-k},\nabla\delta u_i)-(K_k u_{i-k},\delta g_i)_\Gamma \right]\tau^2,\]
which gives
\[
	\begin{array}{rlll}\D
	\abs{\sum_{i=1}^j\sum_{k=1}^i (K_k u_{i-k},\Delta \delta u_i)\tau^2 }
	&\D \le  \sum_{i=1}^j\sum_{k=1}^i\abs{K_k}\vnorma{\nabla u_{i-k}}\vnorma{\nabla\delta u_i}\tau^2 + \sum_{i=1}^j\sum_{k=1}^i\abs{K_k}\vnorma{u_{i-k}}_\Gamma \vnorma{\delta g_i}_\Gamma \tau^2
	\\&\D \le C\sum_{i=1}^j \vnorma{\nabla\delta u_i}\tau + C\sum_{i=1}^j \vnorma{\delta g_i}_\Gamma \tau
	\\&\D\le C_\veps + \veps \sum_{i=1}^j \vnorma{\nabla\delta u_i}^2\tau 
	\end{array}
\]
as Proposition~\ref{Ki_bound}, the trace theorem, $u_i$ and $\nabla u_i$ are $\lp{2}$-bounded (Proposition~\ref{gradui_bound}). The right-hand side of \refe{Plapuit} is rewritten by summation by parts as
\[\sum_{i=1}^j(f_{i-1},\delta \Delta u_i)\tau = \sum_{i=1}^j(f_{i-1},\Delta u_i -\Delta u_{i-1})= (f_{j-1},\Delta u_j )-(f_{0},\Delta u_0)-\sum_{i=1}^{j-1}(\delta f_i ,\Delta u_i)\tau\]
so
\[
	\begin{array}{rlll}\D
		\abs{\sum_{i=1}^j(f_{i-1},\Delta \delta u_i\tau)}
		&\D \le  C_{\veps} + \veps \vnorma{\Delta u_j}^2 +\veps \sum_{i=1}^{j-1}\vnorma{\delta f_{i}}^2\tau + C_{\veps}\sum_{i=1}^{j-1} \vnorma{\Delta u_i}^2\tau
		\\&\D \le  C_{\veps} + \veps \vnorma{\Delta u_j}^2 +\veps \sum_{i=1}^{j-1}\vnorma{\delta u_{i}}^2\tau +\veps \sum_{i=1}^{j-1}\vnorma{\delta \nabla u_{i}}^2\tau + C_{\veps}\sum_{i=1}^{j-1} \vnorma{\Delta u_i}^2\tau
		\\&\D \le  C_{\veps} + \veps \vnorma{\Delta u_j}^2 +\veps \sum_{i=1}^{j-1}\vnorma{\delta \nabla u_{i}}^2\tau + C_{\veps}\sum_{i=1}^{j-1} \vnorma{\Delta u_i}^2\tau
	\end{array}
\]
as $u_0\in\hk{2}$, $\delta f_{i}=\nabla f_{i}\cdot\langle\delta u_{i},\delta\nabla u_{i}\rangle$ and $f$ is Lipschitz.

Putting all things together we arrive at 
\[
	(1-\veps)\sum_{i=1}^j\vnorma{\delta \nabla u_i}^2\tau +\left(\half -\veps\right) \vnorma{\Delta u_j}^2+ \half \sum_{i=1}^k \vnorma{\Delta u_i-\Delta u_{i-1}}^2\le  C_{\veps} + C_{\veps}\sum_{i=1}^{j} \vnorma{\Delta u_i}^2\tau. 
\]
Choosing a suitable $\veps>0$ we close the proof by Gr\"{o}nwall's argument. 

$(ii)$
The relation \refe{eq:DPi} gives for any $\phi\in\hk{1}$ 
\[ 
	\scal{\delta u_i}{\phi}   = \scal{f_{i-1}}{\phi}  - K_i\scal{h_i}{\phi} -\scal{\sum_{k=1}^i K_k u_{i-k}\tau}{\phi} + \scal{\Delta u_i}{\phi}.
\]
The stability results from Propositions~\ref{Ki_bound}--\ref{lapui_bound}$(i)$ ensure that the right-hand side can be seen as a linear bounded functional on $\lp{2}$. Thus, the left-hand side allows extension  from $\hk{1}$ to $\lp{2}$ with the same norm estimate through the Hahn-Banach theorem 
 (cf.  the deduction after \refe{eq:DPi}), i.e.,
 \[\vnorma{\delta u_i} =  \sup_{\vnorma{\phi}\le 1}\abs{(\delta u_i,\phi)}\le  \vnorma{\Delta u_i-K_i h_i -\sum_{k=1}^iK_k u_{i-k}\tau+f_{i-1}}\le  C .\]
\end{proof}

\begin{lem}
	\label{dKi_bound} 
	Assume that  $g\in\ckIX{1}{\lpG{2}}$, $h\in\ckIX{}{\hk{1}}\cap \ckIX{1}{\lp{2}}$, $ \min_{t\in\I} | (h(t),1)| \ge\omega>0$, $u_0\in\hk{2}$ and $m\in\conTk{2}$. 
	The function  $f$ is supposed to be bounded, i.e. $|f|\le C$, and  Lipschitz continuous in all variables. 
	Then there exist $C>0$ such that for any $\tau<\tau_0$   we have
	\[\sum_{i=1}^j\abs{\delta K_i}^2 \tau \le C.\]
\end{lem}
\begin{proof}
The fact that $u_0\in\hk{2}$ implies that the PDE from \refe{problem} is fulfilled at $t=0$, i.e. one can define the initial value for $\partial_t u$ in the following way
\[\partial_t u(0):= f(u_0,\nabla u_0) + \Delta u_0 - K(0) h(0)\in\lp{2}.\]
Applying measurement to this equation gives
\[m'_0+(g_0,1)_{\Gamma}+K_0(h_0,1)=(f_{0},1).\tag{DMP0}\label{DMP0}\]
We would like to apply the $\delta$-operator to \refe{DMPi}. Using the rule $\delta(a_ib_i) = \delta a_i\ b_i + a_{i-1}\ \delta b_i$ we get for $i\ge 2$
\[\delta m'_i+(\delta g_i,1)_{\Gamma}+\delta K_i\ (h_i,1)+K_{i-1}(\delta h_i,1)+K_i m_0+ \sum_{k=1}^{i-1} K_k \delta m_{i-k}\tau=(\delta f_{i-1},1).\]
Thus for $i\ge 2$ it holds 
\[
	\abs{\delta K_i}\abs{(h_i,1)}\le \abs{\delta m'_i}+\abs{(\delta g_i,1)_{\Gamma}}+\abs{K_{i-1}(\delta h_i,1)}+\abs{K_i m_0}+\sum_{k=1}^{i-1} \abs{K_k \delta m_{i-k}}\tau+\abs{(\delta f_{i-1},1)}
	\le  C+C\left(\vnorma{\delta u_{i-1}}+\vnorma{\delta\nabla u_{i-1}}\right).
\]
Further,  we subtract \refe{DMP0} from \refe{DMPi} for $i=1$ to get
\begin{equation}
	\label{eq:i0}
	\delta m'_1+(\delta g_1,1)_{\Gamma}+\delta K_1\ (h_1,1)+K_{0}(\delta h_1,1)+K_1 m_0 =0 
\end{equation}
and we estimate
\[ \abs{\delta K_1}\ \abs{(h_1,1)} \le \abs{K_1 m_0} + \abs{K_{0}(\delta h_1,1)} +\abs{(\delta g_1,1)_{\Gamma}} + \abs{\delta m'_1}.\]
The proof is completed by applying Propositions \ref{gradui_bound} and \ref{lapui_bound} to
\[\sum_{i=1}^j\abs{\delta K_i}^2\tau\le  C+C\sum_{i=2}^j\left(\vnorma{\delta u_{i-1}}^2+\vnorma{\delta\nabla u_{i-1}}^2\right)\tau\le C\]
as $\abs{(h_i,1)}\ge\omega>0$.
\end{proof}

\section{Existence of a solution}
\label{sec:existence}
Now, let us introduce the following piecewise linear function in time 
\[u_n:\I\to\lp{2}:t\mapsto\begin{cases}u_0 & t=0\\
u_{i-1} + (t-t_{i-1})\delta u_i &
     t\in (t_{i-1},t_i]\end{cases},\quad 0\le  i\le  n,\]
and a step function
\[\bar u_n:\I\to\lp{2}:t\mapsto\begin{cases}u_0 & t=0\\
 u_i &     t\in (t_{i-1},t_i]\end{cases},\quad 0\le  i\le  n.\]
 Similarly we define $\bar K_n$, $\bar h_n$, $\bar g_n$, $\bar m_n$ and $\overline{m'}_n$. These prolongations are also called Rothe's (piecewise linear and continuous, or piecewise constant)  functions.
Now, we can rewrite  \refe{DPi} and \refe{DMPi} on the whole time frame as\footnote{${\lfloor t\rfloor_\tau}=i$ when $t\in(t_{i-1},t_i]$}
\[(\partial_t u_n,\phi)+(\nabla \bar u_n,\nabla\phi)+(\bar g_n,\phi)_{\Gamma}+\bar K_n(\bar h_n,\phi)+\sum_{k=1}^{\lfloor t\rfloor_\tau} (\bar K_n(t_k) \bar u_n(t-t_k)\tau,\phi)=(f(\bar u_n(t-\tau),\nabla\bar u_n(t-\tau)),\phi).\tag{DP}\label{DP}\]
and
\[\overline{m'}_n+(\bar g_n,1)_{\Gamma}+\bar K_n(\bar h_n,1)+\sum_{k=1}^{\lfloor t\rfloor_\tau} \bar K_n(t_k) \bar m_n(t-t_k)\tau=(f(\bar u_n(t-\tau),\nabla\bar u_n(t-\tau)),1).\tag{DMP}\label{DMP}\]

Now, we are in a position to prove the existence of a weak solution to \refe{P} and \refe{MP}.
\begin{stel} 
	\label{thm:existence}
	Suppose the conditions of Proposition~\ref{dKi_bound} are fulfilled.
	Then there exists a weak solution $\langle u,K\rangle$ to \refe{P} and \refe{MP}, where
	$u\in\ckIX{}{\hk{1}}$, $\partial_t u \in \Leb^\infty\left(\Iopen,\lp{2}\right)$, $K\in \conTk{}$, $K'\in \Leb^2(0,T)$.
\end{stel}
\begin{proof}
From Propositions \ref{ui_bound} and  \ref{gradui_bound} we have $\vnorma{u_j} \le  C$ and $\sum_{i=1}^n\vnorma{\delta u_i}^2\tau \le  C$, which means that for all $n>0$ it holds
\[\vnorma{\bar u_n(t)}_{\Hi^1(\Omega)}\le  C\quad \text{for all}\quad t\in\I,\quad \int_0^T\vnorma{\partial_t u_n(\xi)}^2\drm \xi\le  C.\]
Using \cite[Lemma 1.3.13]{kacur} there exists $u\in\ckIX{}{\lp{2}}\cap \Leb^\infty\left(\Iopen,\hk{1}\right)$ which is time-differentiable a.e. in $\I$ and a subsequence $(u_{n_k})_{k\in\NN}$ of $(u_n)_{n\in\NN}$ such that
\begin{numcases}{}
	u_{n_k}\to u,&in \quad $\ckIX{}{\lp{2}}$ \label{unktou}\\
	u_{n_k}(t)\rightharpoonup u(t),& in \quad $\hk{1},\quad \forall t\in\I$ \label{unkwtou}\\
	\bar u_{n_k}(t)\rightharpoonup u(t),& in \quad $\hk{1},\quad \forall t\in\I$ \label{bunkwtou}\\
	\partial_t u_{n_k}\rightharpoonup \partial_t u,&in \quad $\Leb^2\left(\Iopen,\lp{2}\right)$\label{dunkwtou}
\end{numcases}
which we denote again by $u_n$ for ease of reading.  Moreover since $\vnorma{\delta u_j} \le  C$ we have that $\partial_t u\in \Leb^\infty\left(\Iopen,\lp{2}\right) $ and $u:\I\to\lp{2}$ is Lipschitz continuous, i.e. $\vnorma{u(t)-u(t')}\le  C\abs{t-t'}$ for all $t,t'\in\I $.

Using Ne\v{c}as' inequality \cite{necas}, the fact that $\sum_{i=1}^n\vnorma{\nabla u_i}^2\tau $ is bounded (Proposition~\ref{ui_bound}) and $u\in \Leb^\infty\left(\Iopen,\hk{1}\right)$ we obtain
\[\int_0^T\vnorma{\bar u_n-u}^2_\Gamma \drm\xi\le  \veps \int_0^T\vnorma{\nabla(\bar u_n-u)}^2\drm\xi+C_\veps \int_0^T\vnorma{\bar u_n-u}^2\drm\xi\le  \veps+C_\veps \int_0^T\vnorma{\bar u_n-u}^2\drm\xi.\]
Passing to the limit and applying \refe{unktou} it holds
\begin{equation}
	\lim_{n\to+\infty}\int_0^T\vnorma{\bar u_n-u}^2_\Gamma \drm\xi\le  \veps \implies \bar u_n\to u, \quad \text{a.e. in }\quad \Iopen\times\Gamma.
	\label{untougam}
\end{equation}
The trace theorem and  Proposition \ref{lapui_bound}  give
\begin{multline*}
		\vnorma{u_n(t+\veps)-u_n(t)}_\Gamma  =\vnorma{\int_t^{t+\veps} \partial_t u_n(s)\di s}_\Gamma \le \int_t^{t+\veps} \vnorma{\partial_t u_n(s)}_\Gamma\di s 
		 \le \sqrt{\veps}\sqrt{ \int_t^{t+\veps} \vnorma{\partial_t u_n(s)}_\Gamma^2\di s } \\
		 \le \sqrt{\veps}\sqrt{ \int_0^T \vnorma{\partial_t u_n(s)}_\Gamma^2\di s } 
		 \le C\sqrt{\veps}\sqrt{ \int_0^T \vnorma{\partial_t u_n(s)}_{\hk{1}}^2\di s } \le C \sqrt{\veps}.
\end{multline*}
This together with \refe{untougam} yields that $u_n\to u$ in $\ckIX{}{\lpG{2}}$.
Integration by parts implies
\[\vnorma{\nabla \bar u_n-\nabla \bar u_m}^2=\scal{-\Delta (\bar u_n-\bar u_m)}{\bar u_n-\bar u_m} - \scal{\bar g_n-\bar g_m}{\bar u_n-\bar u_m}_\Gamma.\]
The assumption $g\in\ckIX{1}{\lpG{2}}$ ensures that $\bar g_n \to g$ in $\ckIX{}{\lpG{2}}$. Taking into account $\vnorma{ \Delta \bar u_n(t)}\le C$, cf. Proposition \ref{lapui_bound}, we see that $\bar u_n$ is a Cauchy sequence in $\ckIX{}{\hk{1}}$. 
Combining this with 
\[
	\vnorma{u_n(t)-\bar u_n(t)}_{\hk{1}}\le C\sqrt{\tau}\sqrt{\int_t^{t+\tau}\vnorma{\partial_t u_n(s)}^2_{\hk{1}}\di s}\le C\sqrt{\tau}
\]
and \refe{unktou} we conclude that
\begin{equation}
	\label{eq:gradun}
	\bar u_n\to u, \quad u_n\to u \qquad \mbox{ in } \ckIX{}{\hk{1}}.
\end{equation}
Using Propositions~\ref{Ki_bound} and \ref{dKi_bound} we have
\[\abs{\bar K_n(t)}\le  C\quad \text{for all}\quad t\in\I,\quad \int_0^T\abs{\partial_t K_n(\xi)}^2\drm \xi\le  C,\]
which means by the Arzel\`a-Ascoli theorem \cite[Theorem 11.28]{rudin} that there exists a subsequence $(K_{n_k})_{k\in\NN}$ (which we denote by the same symbol again) that converges uniformly on $\I$, say to $K$.
The reflexivity of the space $\Leb^2(0,T)$ implies that $\partial_t K_n\weak \partial_t K$ in $\Leb^2(0,T)$.
It becomes trivial to see that
\begin{equation}
	\lim_{n\to+\infty}\int_0^t\bar K_n(\bar h_n,\phi)\drm\xi=\int_0^t K(h,\phi)\drm\xi.
	\label{term3en4}
\end{equation}
Applying \refe{unktou} combined with the uniform convergence of $K_n$ we have
\begin{equation}
	\lim_{n\to+\infty}\int_0^t\sum_{k=1}^{\lfloor t\rfloor_\tau} (\bar K_n(t_k) \bar u_n(\xi-t_k)\tau,\phi)\drm\xi=\int_0^t(K\ast u,\phi)\drm\xi.
	\label{term5}
\end{equation}

Now, when we integrate \refe{DP} and let $n\to+\infty$ ($\tau\to0$) we obtain the following limit for the l.h.s. 
\[\int_0^t(\partial_t u,\phi)\drm\xi+\int_0^t(\nabla u,\nabla\phi)\drm\xi+\int_0^t( g,\phi)_{\Gamma}\drm\xi+\int_0^t K( h,\phi)\drm\xi+\int_0^t(K\ast u,\phi)\drm\xi.\]
First, by \refe{eq:gradun} we have
\[\lim_{n\to+\infty}\int_0^t\vnorma{f(\bar u_n(\xi),\nabla\bar u_n(\xi))-f(u(\xi),\nabla u(\xi))}\drm\xi=0.\]
Secondly, as $f$ is Lipschitz we have
\[
	\vnorma{f(\bar u_n(\xi-\tau),\nabla\bar u_n(\xi-\tau))-f(\bar u_n(\xi),\nabla\bar u_n(\xi))}\]
	\[\le  C \sqrt{\vnorma{\bar u_n(\xi-\tau)-\bar u_n(\xi)}^2+\vnorma{\nabla\bar u_n(\xi-\tau)-\nabla\bar u_n(\xi)}^2}
	=C \tau \sqrt{\vnorma{\partial_t u_n(\xi)}^2+\vnorma{\nabla \partial_tu_n(\xi)}^2}=C\tau\vnorma{\partial_t u_n(\xi)}_{\hk{1}}
\]
from which it follows that
\[
	\lim_{n\to+\infty}\abs{\int_0^t(f(\bar u_n(\xi-\tau),\nabla\bar u_n(\xi-\tau))-f(u(\xi),\nabla u(\xi)),\phi)\drm\xi }^2
	\le \vnorma{\phi}^2 C\lim_{n\to+\infty}\tau\int_0^t\vnorma{\partial_t u_n(\xi)}_{\hk{1}}^2\drm\xi=0
\]
as from Propositions~\ref{gradui_bound} and \ref{lapui_bound} we have
\[ \int_0^T\vnorma{\partial_t u_n(\xi)}^2\drm \xi+\int_0^T\vnorma{\nabla \partial_t u_n(\xi)}^2\drm \xi\le  C.\]
From the above we conclude that for the integrated r.h.s. of \refe{DP} it holds
\[\lim_{n\to+\infty}\int_0^t(f(\bar u_n(\xi-\tau),\nabla\bar u_n(\xi-\tau)),\phi)\drm\xi=\int_0^t(f(u(\xi),\nabla u(\xi)),\phi)\drm\xi.\]
We conclude that taking the limit for $n\to+\infty$ ($\tau\to0$) in \refe{DP} results in
\[\int_0^t(\partial_t u,\phi)\drm\xi+\int_0^t(\nabla u,\nabla\phi)\drm\xi+\int_0^t( g,\phi)_{\Gamma}\drm\xi+\int_0^t K( h,\phi)\drm\xi+\int_0^t(K\ast u,\phi)\drm\xi=\int_0^t(f(u(\xi),\nabla u(\xi)),\phi)\drm\xi.\]
Taking the derivative with respect to $t$ we arrive at \refe{P}. 

Finally, we have to pass to the limit for $\tau\to 0$ in \refe{DMPi} to arrive at \refe{MP}. This follows the same line as passing the limit in \refe{DPi}, therefore we skip the details.
\end{proof}

The convergences of Rothe's functions towards the weak solution \refe{P}-\refe{MP} (as stated in the proof of Theorem \ref{thm:existence}) have been shown for a subsequence. 
Note, that taking into account Theorem \ref{thm:uniqueness} we see that the whole Rothe's functions converge against the solution.

\section*{Conclusion}

A semilinear parabolic integro-differential problem of second order with an unknown convolution kernel is considered. The existence and uniqueness of a weak solution for the IBVP is proved. The missing integral kernel is recovered from an integral-type measurement.  
A numerical algorithm based on Rothe's method is established and the convergence of approximations towards the exact solution is demonstrated. 



 
\section*{Acknowledgment}
The research was supported by the IAP P7/02-project of the Belgian Science Policy.

\bibliographystyle{elsarticle-num}
\bibliography{JDE_rds}







\end {document}